\documentclass[11pt,reqno]{amsart}
\usepackage{amsfonts}
\usepackage{amsthm}
\usepackage{amsmath}
\usepackage{amssymb}
\usepackage[dvips]{graphics,graphicx}
\usepackage{caption}
\usepackage[T1]{fontenc}
\usepackage{epsfig}
\usepackage{lmodern}
\usepackage{exscale}
\usepackage[latin1]{inputenc}
\usepackage{xcolor}

\usepackage{fancyhdr}
\usepackage{ifthen}
\usepackage{geometry}
 \selectfont
\theoremstyle{plain}
\newtheorem{thm}{Theorem}[section]
\newtheorem{pro}{Proposition}[section]
\newtheorem{lem}{Lemma}[section]

\newtheorem{re}{Remark}[section]
\newtheorem{res}{Remarks}[section]

\def\eps{\varepsilon}

\title[Single-point blow-up for systems]{single-point blow-up for
parabolic systems with exponential nonlinearities and unequal diffusivities}
\author[Ph. Souplet]{Philippe Souplet}
\address{Universit\'e Paris 13, Sorbonne Paris Cit\'e, CNRS UMR 7539 LAGA, 99,
Avenue Jean-Baptiste Cl\'ement, 93430 Villetaneuse,
France.}\email{souplet@math.univ-paris13.fr}
\author[S. Tayachi]{Slim Tayachi}
\address{Universit\'e de Tunis El Manar, Facult\'e des Sciences de Tunis, D\'epartement de
Math\'ematiques, Laboratoire  \'Equations aux D\'eriv\'ees
Partielles LR03ES04,  2092 Tunis,
Tunisie.}\email{slim.tayachi@fst.rnu.tn} 
\begin{document}
\begin{abstract}  We study positive blowing-up
solutions of systems of the form:
$$u_t=\delta_1 \Delta u+e^{pv},\quad v_t= \delta_2\Delta v+e^{qu},$$
with $\delta_1,\delta_2>0$ and $p, q>0$.
We prove single-point blow-up for large classes of radially
decreasing solutions.
This answers a question left open in a paper of Friedman and Giga~(1987),
where the result was obtained only for the equidiffusive case $\delta_1=\delta_2$
and the proof depended crucially on this assumption.
\end{abstract}
\maketitle

\section{Introduction}
 In this paper, we consider nonnegative solutions of the following
reaction-diffusion system:
\begin{equation}
\label{e1}
\begin{cases}
u_t=\delta_1 \Delta u+f(v), & t>0,\ x\in\Omega,\\
v_t= \delta_2\Delta v+g(u), & t>0,\ x\in\Omega,\\
\end{cases}
\end{equation}
with possibly unequal diffusivities
$\delta_1, \delta_2>0$,  and nonlinearities of exponential type, namely:
\begin{equation}
\label{efgm=0}
f(v)=e^{pv},\quad g(u)=e^{qu},\quad p,\; q> 0
\end{equation}
or
\begin{equation}
\label{efgm=1}
f(v)=e^{pv}-1,\quad g(u)=e^{qu}-1,\quad p,\; q> 0.
\end{equation}

Problem (\ref{e1}) is a basic model case for reaction-diffusion systems
and, as such, it has been the subject of intensive investigation for more than 20 years
(see e.g. \cite[Chapter~32]{pavol} and the references therein).
We are here mainly interested in proving
single-point blow-up for system (\ref{e1}) with exponential nonlinearites.

\smallskip
For system (\ref{e1}) and with $f,\; g$ given by \eqref{efgm=0},
the blow-up set was first studied in~\cite{giga}. In that work,
Friedman and Giga considered symmetric nonincreasing solutions of the one-dimensional initial-Dirichlet problem
and, under the restrictive condition $\delta_1=\delta_2$, they proved that blow-up occurs only at the origin; see \cite[Theorem 3.1, p. 73]{giga}.
Note that the assumption $\delta_1=\delta_2$ is essential in \cite{giga} in order to apply the maximum principle to
suitable linear combinations of the components $u$ and $v$, so as to
derive comparison estimates between them. The problem for $\delta_1\not=\delta_2$
was left open. In this instance, we recall that non-equidiffusive parabolic systems
are often much more involved, both in terms of behavior of solutions
and at the technical level (cf.~\cite{pierre} and \cite[Chapter~33]{pavol}).

\smallskip
The purpose of this paper is to give an answer to this question.
We will actually consider more generally the radially symmetric problem in higher dimensions.
In what follows, for $R\in (0,\infty]$, we denote $B_R=\left\{x\in\mathbb{R}^n\; ;|x|<R\right\}$, with $n\ge 1$ an integer
 (so, $B_R=\mathbb{R}^n$ for $R=\infty$).
We say that $(u,v)$ is radially symmetric nonincreasing if
\begin{equation}\label{monot}
  \begin{array}{ll}
  u=u(t,\rho),\ v=v(t,\rho)\quad\hbox{with $\rho=|x|$, } \\
  \noalign{\vskip 1mm}
 u_\rho,\ v_\rho\le 0\quad\hbox{for $0<t<T$ and $0<\rho<R$.}
  \end{array}
\end{equation}
Our main result is the following.

\begin{thm}[Single-point blow-up] \label{th1}
Let $\delta_1, \delta_2>0$, $T\in (0,\infty)$, $R\in (0,\infty]$ and $\Omega= B_R$. Let $f,\; g$ be given by \eqref{efgm=0} or  \eqref{efgm=1}.
Let $(u,v)$ be a nonnegative, radially symmetric nonincreasing, classical solution of (\ref{e1}) in $(0,T)\times\Omega$,
with $u_\rho\not\equiv 0$ or $v_\rho\not\equiv 0$.
Assume that $(u,v)$ satisfies the type~I blow-up estimates:
\begin{equation}
\label{estimateTypeI}
 q\|u(t)\|_{L^\infty(\Omega)}\leq |\log(T-t)|+C,\quad p\|v(t)\|_{L^\infty(\Omega)}\leq |\log(T-t)|+C,\quad 0<t<T,
\end{equation}
for some constant $C>0$. Then blow-up occurs only at the origin, i.e.:
\begin{eqnarray*}
\underset{0<t<T}{\sup}(u(t,\,\rho)+v(t,\,\rho))<\infty,\quad\hbox{for
all}\,\,\rho\in(0,R).
\end{eqnarray*}
\end{thm}

In order to produce actual solutions with single-point blow-up,
we of course need to consider initial-boundary value problems associated with system (\ref{e1}),
and in particular we have to ensure the type I blow-up assumption \eqref{estimateTypeI}.
For $\Omega\subset \mathbb{R}^n$ a smooth bounded domain, we consider the Dirichlet problem
\begin{equation}
\label{e1DP}
\begin{cases}
u_t=\delta_1 \Delta u+f(v), & t>0,\ x\in \Omega,\\
v_t= \delta_2\Delta v+g(u), & t>0,\ x\in \Omega,\\
u(t,x)=v(t,x)=0, & t>0,\ x\in \partial \Omega,\\
u(0,x)=u_{0}(x),\ v(0,x)=v_{0}(x), &x\in \Omega,\\
\end{cases}
\end{equation}
and the Neumann problem
\begin{equation}
\label{e1NP}
\begin{cases}
u_t=\delta_1 \Delta u+f(v), &  t>0,\ x\in \Omega,\\
v_t= \delta_2\Delta v+g(u), &  t>0,\ x\in \Omega,\\
{\partial u\over \partial \nu}(t,x)={\partial v\over \partial \nu}(t,x)=0, & t>0,\ x\in \partial \Omega, \\
u(0,x)=u_{0}(x),\ v(0,x)=v_{0}(x), &x\in \Omega,\\
\end{cases}
\end{equation}
where $\nu =\nu(x)$ denotes the unit outer normal
and the initial data are assumed to satisfy
\begin{equation}\label{conditioninitialdata}
  u_{0}, v_{0}\in L^{\infty}(\Omega), \quad u_{0}, v_{0}\ge 0.
  \end{equation}
For $\Omega= \mathbb{R}^n$, we also consider the Cauchy problem
\begin{equation}
\label{e1CP}
\begin{cases}
u_t=\delta_1 \Delta u+f(v), & t>0,\ x\in \mathbb{R}^n,\\
v_t= \delta_2\Delta v+g(u), &  t>0,\ x\in \mathbb{R}^n,\\
u(0,x)=u_{0}(x),\ v(0,x)=v_{0}(x), &x\in \mathbb{R}^n.
\end{cases}
\end{equation}

Under assumptions (\ref{conditioninitialdata}) and with \eqref{efgm=0} or \eqref{efgm=1}, each of problems
(\ref{e1DP}), (\ref{e1NP}) and (\ref{e1CP}) has a unique nonnegative, maximal solution $(u,v)$, classical for $t>0$.
The maximal existence time of $(u,v)$ is denoted by $T\in (0,\infty]$.
If moreover $T<\infty$, then
 \begin{equation*}
\limsup_{t \rightarrow T}\,(\|u(t)\|_{\infty}+\|v(t)\|_{\infty})=\infty,
\end{equation*}
and we say that the solution blows up in finite time with blow-up
time $T$.
In case $\Omega=B_R$ ($R\le \infty$), if in addition
\begin{equation}\label{conditioninitialdata2}
\hbox{ $u_{0}, v_{0}$ is radially symmetric, radially nonincreasing,}
\end{equation}
then $(u,v)$ is radially symmetric nonincreasing.

\smallskip

The following Theorem gives sufficient conditions for the existence of blow-up solutions satisfying the type~I estimates \eqref{estimateTypeI}
(in general domains) and therefore, as a consequence of Theorem~\ref{th1}, for single-point blow-up in the symmetric case.
We point out that the type I estimate in Theorem~\ref{th2}(i) is actually a consequence of more general results for nonequidiffusive systems, that we obtain in Section 5 below.

\begin{thm}
\label{th2}
Let $\delta_1,\delta_2>0$. Let $\Omega\subset \mathbb{R}^n$ be a smooth bounded domain or $\Omega= \mathbb{R}^n$.
Assume that one of the following three assumptions holds:

(a) $(u,v)$ is the solution of the Dirichlet problem (\ref{e1DP}) with $f,\; g$ given by \eqref{efgm=1}, and $(u_0,v_0)$ satisfying
\begin{equation*}
\left\{
  \begin{array}{ll}
    \mbox{$u_0,\,v_0\in C^2(\Omega)\cap C(\overline\Omega),\quad u_0,\,v_0\ge 0,\quad  u_0=v_0=0$ on $\partial\Omega$}, & \hbox{ } \\
    \noalign{\vskip 1mm}
    \delta_1 \Delta u_0+f(v_0)\geq 0, \quad
    \delta_2 \Delta v_0+g(u_0)\geq 0\quad\mbox{ in $\Omega$};
   \end{array}
\right.
\end{equation*}

(b) $(u,v)$ is the solution of the Neumann problem (\ref{e1NP}) with $f,\; g$ given by \eqref{efgm=0} or \eqref{efgm=1}, and $(u_0,v_0)$ satisfying
\begin{equation*}
\left\{
  \begin{array}{ll}
    \mbox{$u_0,\,v_0\in C^2(\Omega)\cap C^1(\overline\Omega),\quad u_0,\,v_0\ge 0,\quad   {\partial u_0\over \partial \nu}={\partial v_0\over \partial \nu}\le 0$ on $\partial\Omega$}, & \hbox{ } \\
    \noalign{\vskip 1mm}
    \delta_1 \Delta u_0+f(v_0)\geq 0, \quad
    \delta_2 \Delta v_0+g(u_0)\geq 0\quad\mbox{ in $\Omega$};
     \end{array}
\right.
\end{equation*}

(c) $(u,v)$ is the solution of the Cauchy problem (\ref{e1CP}) with $f,\; g$ given by \eqref{efgm=0} or \eqref{efgm=1}, and $(u_0,v_0)$ satisfying,
for some $\eps\in (0,1)$,
\begin{equation*}
\left\{
  \begin{array}{ll}
    \mbox{$u_0,\,v_0\in BC^2(\mathbb{R}^n),$}\quad u_0,\,v_0\ge 0, & \hbox{ } \\
    \noalign{\vskip 1mm}
    \delta_1 \Delta u_0+(1-\eps)f(v_0)\geq 0, \quad
    \delta_2 \Delta v_0+(1-\eps)g(u_0)\geq 0\quad\mbox{ in $\mathbb{R}^n$}.
  \end{array}
\right.
\end{equation*}
Assume in addition that
$\delta_1 \Delta u_0+f(v_0)\not\equiv 0$ or $\delta_2 \Delta v_0+g(u_0)\not\equiv 0$.
Then:

(i) we have $T=T(u_0,v_0)<\infty$ and the type~I blow-up estimate \eqref{estimateTypeI} is satisfied;

(ii) in the case $\Omega=B_R$ or $\Omega=\mathbb{R}^n$, under the additional assumption (\ref{conditioninitialdata2})
with $(u_{0}, v_{0})$ nonconstant, 
the solution $(u,v)$ blows up only at the origin.
\end{thm}

\begin{res}\rm{
(a) The assumption $u_\rho\not\equiv 0$ or $v_\rho\not\equiv 0$ in Theorem~\ref{th1}
(or $(u_0,v_0)$ nonconstant
in Theorem~\ref{th2}) 
is necessary,
due to the existence of spatially homogeneous solutions (for the Neumann and Cauchy problems), which blow up at every point.

(b) Due to the nature of the coupling in system (\ref{e1}), it is easy to see that blow-up is always simultaneous
(i.e. both $u$ and $v$ become unbounded as $t\to T<\infty$).

(c) Theorem~\ref{th1} remains true if the upper type I estimate \eqref{estimateTypeI} is only assumed away from the origin
(this follows from the proof of Theorem~\ref{th1} in view of Proposition~\ref{15}).
However, we do not know if this property can be guaranteed without assuming the conditions in Theorem~\ref{th2}.
 The type I estimate (away from the origin) is crucial to our analysis
of the more involved case $\delta_1\neq\delta_2$ (see the outline of proof in section~2).
Nevertheless, we note that in the case of systems with power-type nonlinearities \cite{MST},
the type I estimate away from the origin is actually known to be true for any radially nonincreasing solution
(cf.~Remark~\ref{rem51a}(a) below).
In particular the result in \cite{MST} on single-point blow-up for systems with power-type nonlinearities,
requires no type I assumption at all.}

\end{res}

\begin{res}\rm{
(a) The assumptions in Theorem~\ref{th2} guarantee that the solution is monotone in time.
It seems that type I blow-up estimates for monotone in time solutions of reaction-diffusion systems are only known
in the equidiffusive case.
See \cite{deng} for power nonlinearities and \cite{LinWang} for exponential nonlinearities.
These results are based on the well-known maximum principle technique introduced in \cite{friedman} for scalar equations
(see also \cite{Sp}). In order to cover the nonequidiffusive case, we here need to introduce
a not completely trivial modification of this method.
Our arguments in fact work for general nonlinearities and systems of any number of unknowns
(see section 5).

(b) It is still a widely open problem how to obtain the type I estimate in the case of exponential nonlinearities
when the solutions are not monotone in time.
Even in the scalar case, the only result in that direction seems to be that in \cite{FiPul},
which concerns radially decreasing solutions in dimensions
$n\in [3,9]$.\footnote{\ For the delicate role of the space dimension in the scalar problem with exponential nonlinearity,
see \cite{Vaz} and the references therein.}
But the proof, based on zero-number arguments, does not extend to systems.

(c) In case (a) of Theorem~\ref{th2}, for the Dirichlet boundary conditions,
the conclusions remain true for the nonlinearities in \eqref{efgm=0},
provided we know that the blow-up set is a compact subset of $\Omega$. However we do not know presently
how to ensure this condition.}
\end{res}

\section{Outline of proof of Theorem~\ref{th1}.}
\setcounter{equation}{0}

As mentioned above, the study of single-point blow-up for parabolic systems was initiated in \cite{giga},
where the problem was also studied for power-type nonlinearities, typically
\begin{equation}\label{powerNL}
f(v)=v^p,\quad g(u)=u^q.
\end{equation}
The basic idea, introduced in \cite{giga} (extending a method from \cite{friedman} for scalar equations),
 is to consider auxiliary functions $J,\,\overline{J}$ of the form:
\begin{equation}\label{choiceJgiga}
J(t,\,\rho)=u_\rho+\eps c(\rho)F(u),\quad
\overline{J}(t,\,\rho)=v_\rho+\eps \overline c(\rho) G(v).
\end{equation}
The couple $(J,\,\overline{J})$ satisfies a system of parabolic inequalities
to which one aims at applying the maximum principle, so as to deduce that $J, \overline{J}\leq 0$.
By integrating these inequalities in space, one then obtains
upper bounds on $u$ and $v$ which guarantee single-point blow-up at the origin.

However, in the case of systems, such a procedure turns out to require good comparison properties between $u$ and $v$.
The comparison properties employed in \cite{giga} were of global nature
and relied upon an application of the maximum principle to a linear combination of $u$ and $v$,
thus entailing to impose the equidiffusivity condition $\delta_1=\delta_2$.
In the case of power nonlinearities (\ref{powerNL}), this even required the severe restriction $p=q$.
This restriction was later removed in \cite{souplet} to cover the whole range of parameters $p,q>1$,
by applying a different strategy.
Namely, instead of looking for comparison properties valid everywhere,
one assumes for contradiction that single-point blow-up fails; one then establishes sharp asymptotic estimates near
nonzero blow-up points, which in particular yield local comparison properties that turn out to be sufficient to
handle the system satisfied by auxiliary functions similar to those  in~(\ref{choiceJgiga}).
The proof of the sharp asymptotic estimates near nonzero blow-up points is based on
similarity variables, delayed smoothing effects for rescaled solutions, monotonicity arguments and
rigidity properties in connection with an associated ODE system.
The result in \cite{souplet} still required $\delta_1=\delta_2$, along with an assumption of type I blow-up,
but both assumptions were later removed in \cite{MST} by further refinements of the arguments in \cite{souplet}.

We here follow the same basic strategy as in \cite{souplet, MST}. However, in the case of exponential nonlinearities,
specific difficulties appear to establish the lower asymptotic estimates near possible nonzero blow-up points.
This is mainly due to the fact that, unlike for the case of power nonlinearities, the rescaling by similarity variables for exponential nonlinearities
does not preserve positivity and may lead to solutions unbounded from below.
In previous studies of blow-up asumptotics for the scalar equation $u_t-\Delta u=e^u$ (cf. \cite{BBE, BE, FiPul, Pul}),
this was overcome by using the estimate
$$|\nabla u|^2\le 2 e^{\|u(t,\cdot)\|_\infty},$$
established in \cite{friedman} by a maximum principle argument, which restores the local compactness of rescaled solutions.
However, such an estimate does not seem to carry over to the case of systems, especially when $\delta_1\neq\delta_2$.
To circumvent this, we
perform the change of variables $U=pe^{qu},\ V=qe^{pv}$, which converts
solutions of (\ref{e1})  with \eqref{efgm=0} or \eqref{efgm=1} to {\it subsolutions} of the system
\begin{equation}
\label{ineq1a}
\begin{cases}
U_t\leq \delta_1 \Delta U+UV,\\
V_t\leq \delta_2 \Delta V +UV.
\end{cases}
\end{equation}
Under upper type~I blow-up assumption for $(u,v)$, a rescaling of $(U,V)$ by similarity variables
leads to a global-in-time, {\it bounded subsolution} of the system
\begin{equation}
\label{ineq1aWZ}
\begin{cases}
W_s\leq \delta_1 \Delta W-\frac{y}{2}\cdot\nabla W-W+WZ\\  
Z_s\leq \delta_2 \Delta Z-\frac{y}{2}\cdot\nabla Z-Z +WZ.
\end{cases}
\end{equation}
At this point, a further modification of the arguments from \cite{souplet, MST} is necessary. Indeed, in those works, the key nondegeneracy property for nonzero blow-up points is proved in two steps.
A first step is to use delayed smoothing effects to show that if both components should blow up in a degenerate way at a given point and at some time close enough to $T$, then the rescaled solution would decay exponentially
as $s\to\infty$, 
 leading to a contradiction with the blow-up of the original solution at that point.
The second step is to prove that none of the two components can actually degenerate, by showing suitable interdependence of the components.
This second step relies on compactness arguments and rigidity properties in connection with an associated ODE system,
and thus requires to deal with solutions and not mere subsolutions.
To circumvent this, we shall take advantage of the particular product form and of the equality
of the nonlinearities in (\ref{ineq1a}).\footnote{\ See the proof of Proposition~\ref{15} below and especially its step 3.
We stress that this property is quite specific to the exponential case and does not carry over to power type nonlinearities.}
This special structure will allow us to prove nondegeneracy of both components of (\ref{ineq1a}) in a single step,
using delayed smoothing arguments and a careful comparison with a modified solution of (\ref{ineq1aWZ}).

\smallskip

\smallskip

The organization of the rest of this paper is as follows. In Section~3 we prove
 the key nondegeneracy property Proposition~\ref{15}.
In Section 4, we prove Theorem~\ref{th1}. Finally, in Section~5, we give additional type I estimates
for more general problems and prove Theorem~\ref{th2}.

\section{ Non-degeneracy criterion for blow-up points}
\setcounter{equation}{0}

The main objective of this section is a result which gives a sufficient, local smallness condition on a single component,
 at any given time sufficiently close to $T$,
for excluding blow-up of $(u,v)$ at a given point different from the origin.

Let $(u,\; v)$ be a solution of system \eqref{e1}, with $f,g$ given by \eqref{efgm=0} or \eqref{efgm=1}. Put
$$U=pe^{qu},\; V=qe^{pv}.$$
Using the fact that $U_t- \delta_1\Delta U=qU(u_t-\delta_1\Delta u)-\delta_1 U^{-1}|\nabla U|^2$
and the analogous formula for $V$, we see that $U$ and $V$ satisfy
\begin{equation}
\label{e2}
\begin{cases}
U_t-\delta_1 \Delta U\le UV-\delta_1 U^{-1}|\nabla U|^2,\quad  &t>0, \; x\in \Omega, \\
V_t-\delta_2 \Delta V\le UV-\delta_2 V^{-1}|\nabla V|^2,\quad  &t>0, \; x\in \Omega,
\end{cases}
\end{equation}
hence in particular,
\begin{equation}
\label{ineq1}
\begin{cases}
U_t\leq \delta_1 \Delta U+UV,\quad  &t>0, \; x\in \Omega, \\
V_t\leq \delta_2 \Delta V +UV,\quad  &t>0, \; x\in \Omega.
\end{cases}
\end{equation}

\begin{pro}\label{15}
Let $\delta_1, \delta_2 >0$, $T, R\in (0,\infty)$ and $\Omega= B_R$.
Let $(U,V)$ be a nonnegative, radially symmetric, classical solution of (\ref{ineq1}) in $(0,T)\times\Omega$,
such that
\begin{equation}
\label{UrhoVrho}
U_\rho, V_\rho\le 0\quad\hbox{in $(0,T)\times\Omega$. }
\end{equation}
Let $d_0,\,d_1$ satisfy $0<d_1<d_0<R$.
Assume that $(U,V)$ satisfies the upper estimates:
\begin{equation}
\label{estimateTypeIb}
(T-t)U(t,r)\leq M_0,\quad (T-t)V(t,r)\leq M_0,\qquad 0<t<T,\ d_1\le r\le R
\end{equation}
for some constant $M_0>0$.
There exist $\eta,\, \tau_0>0$ such that if, for some
$t_1\in[T-\tau_0,\,T)$, we have
\begin{eqnarray}\label{17}
(T-t_1) U(t_1,\,d_1)\leq \eta
\quad\hbox{ or }\quad
(T-t_1) V(t_1,\,d_1)\leq \eta,
\end{eqnarray}
then $d_0$ is not a blow-up point of
$(U,\,V)$, i.e. $(U,\,V)$ is uniformly bounded in the neighborhood
of $(T,\,d_0).$
Here, the numbers $\eta,\,\tau_0$ depend only on $\delta_1,\delta_2,d_0,\,d_1,\,R,\,T,\, M_0$.
\end{pro}

Before giving the proof of the proposition, we first recall the following
 linear results.

\subsection{Similarity variables and delayed smoothing effects}
In view of the proof of Proposition~\ref{15}
we introduce the well-known similarity variables (cf.~\cite{kohn1}).
More precisely, for any given $d\in\mathbb{R}$, we define the
one-dimensional similarity variables around $(T,\,d)$,
associated with $(t,\, \rho)\in (0,\,T)\times\mathbb{R}$, by:
\begin{equation}\label{defsimilvar}
\sigma=-\log(T-t)\in [\hat\sigma,\,\infty),\qquad
\theta=\frac{ \rho-d}{\sqrt{T-t}}=e^{\sigma/2}( \rho-d) \in \mathbb{R},
\end{equation}
where $\hat\sigma=-\log T$.
 For given $\delta>0$, let $U$ be a (strong) solution of
$$U_t-\delta U_{\rho\rho}=H(t,\,\rho),\quad 0<t<T,\ \rho\in \mathbb{R},$$
 where $H\in L^\infty_{loc}([0,T); L^\infty(\mathbb{R}))$ is a given function.
Then
$$W= W_d(\sigma,\,\theta)=(T-t)U(t,\,y)
 = e^{-\sigma}U\bigl(T-e^{-\sigma},\,d+\theta e^{-\sigma/2}\bigr)$$
is a solution of
\begin{equation}\label{eqsimilV}
 W_\sigma-\mathcal{L_{\delta}} W+ W=
e^{-2\sigma}H\bigl(T-e^{-\sigma},\,d+\theta e^{-\sigma/2}\bigr),
 \quad \sigma>\hat\sigma,\ \theta\in \mathbb{R},
\end{equation}
where
\begin{eqnarray*}
 &&\mathcal{L_{\delta}}=\delta\partial^2_\theta -\frac{\theta }{2}\partial_\theta =\delta K_{\delta}^{-1}\partial_\theta (K_{\delta}\partial_\theta ),\qquad
 K_{\delta}(\theta )=(4\pi\delta)^{-1/2}e^{\frac{-\theta^2}{4\delta}}.
\end{eqnarray*}
We denote by
$(T_{\delta}(\sigma))_{\sigma\geq0}$ the semigroup associated with
$\mathcal{L}_{\delta}$. More precisely, for each $\phi\in
L^\infty(\mathbb{R}),$ we set
$T_{\delta}(\sigma)\phi:=w(\sigma,\,.)$, where $w$ is the unique
solution of
$$
\left\{
  \begin{array}{ll}
    w_\sigma=\mathcal{L}_{\delta}w, & \hbox{ } \theta \in\mathbb{R},\,\,\sigma>0,\\
    w(0,\,\theta )=\phi(\theta ), & \hbox{ }\theta \in\mathbb{R}.
  \end{array}
\right.
$$
For any $\phi\in L^\infty(\mathbb{R}),$ we put
\begin{equation*}
    \|\phi\|_{L_{K_{\delta}}^m}=\left(\int_{\mathbb{R}}|\phi(\theta )|^mK_{\delta}(\theta )d\theta \right)^{1/m},\quad 1\leq m<\infty.
\end{equation*}
 Let $1\leq k<m<\infty$ and $\delta>0$, then,  by Jensen's inequality,
 \begin{eqnarray*}
        & &\|\phi\|_{L_{K_{\delta}}^k}\leq
         \|\phi\|_{L_{K_{\delta}}^m},\quad 1\leq
        k<m<\infty.
     \end{eqnarray*}
 The analysis in \cite{souplet, MST}, inspired by arguments from in \cite{HVihp, AHV},
makes crucial use of delayed smoothing effects for the
semigroups $(T_{\delta}(\sigma))_{\sigma\geq0}$.
Namely, we have the following properties (see, e.g., \cite{MST}).

\begin{lem}\label{66}
\begin{enumerate}
  \item (Contraction) We have
    \begin{equation}\label{11a}
\|T_{\delta}(\sigma)\phi\|_\infty\leq\|\phi\|_\infty,
 \quad\; \mbox{for all}\; \delta>0,\; \sigma\geq0,\,\phi\in L^\infty(\mathbb{R})
\end{equation}
and, for any $1\leq m<\infty$,
  \begin{equation}\label{11}
\|T_{\delta}(\sigma)\phi\|_{L_{K_{\delta}}^m}\leq\|\phi\|_{L_{K_{\delta}}^m},
 \quad\; \mbox{for all}\; \delta>0,\; \sigma\geq0,\,\phi\in L^\infty(\mathbb{R}).
\end{equation}
Moreover,  for all $0<\delta\le \lambda<\infty$, we have
$$
 \|T_\delta(\sigma)\phi\|_{L_{K_\lambda}^m}\leq\Bigl(\frac{\lambda}{\delta}\Bigr)^{1/2}\|\phi\|_{L_{K_\lambda}^m},
 \quad\; \mbox{for all}\; \sigma\geq0,\,\phi\in L^\infty(\mathbb{R}).
$$

  \item (Delayed regularizing effect) For any $1\leq k<m<\infty,$ there exist ${\hat C},$ $\sigma^\ast>0$ such that
$$
\|T_{\delta}(\sigma)\phi\|_{L_{K_{\delta}}^m}\leq \hat C\|\phi\|_{L_{K_{\delta}}^k},
 \quad\; \mbox{for all}\;  \delta>0,\ \sigma\geq \sigma^\ast,\,\phi\in L^\infty(\mathbb{R}).
$$
Moreover,  for all $0<\delta\le \lambda<\infty$, we have
$$
 \|T_\delta(\sigma)\phi\|_{L_{K_\lambda}^m}\leq \hat C\Bigl(\frac{\lambda}{\delta}\Bigr)^{1/2}\|\phi\|_{L_{K_\lambda}^k},
 \quad\; \mbox{for all}\;\sigma\geq \sigma^\ast,\,\phi\in L^\infty(\mathbb{R}).
$$

\end{enumerate}
\end{lem}

We now turn to   
the proof of Proposition~\ref{15}.

\subsection{ Proof of Proposition~\ref{15}}
The proof is somewhat technical. We split it in several steps.

\smallskip
{\bf Step 1.} {\it Definition of suitably  modified solutions.}
Due to $U_\rho,V_\rho\le 0$, the solution $(U,V)$ satisfies
\begin{equation}
\label{ineq1b}
\left\{
  \begin{array}{ll}
U_t\leq \delta_1 U_{\rho\rho}+UV,
    &\, 0<t<T,\ 0<\rho<R,\\
V_t\leq  \delta_2 V_{\rho\rho} +UV,
    &\, 0<t<T,\ 0<\rho<R.\\
\end{array}
\right.
\end{equation}
Since the upper estimate \eqref{estimateTypeIb} is only assumed to hold for $r\ge d_1$,
we shall truncate the radial domain and
consider suitably controlled extensions of the solution to the real
line. We first define the following extensions $\widetilde u, \widetilde v\ge 0$ of $U, V$ by setting:
\begin{equation}\label{defutilde}
\widetilde u(t,\,y):=
\begin{cases}
U(t,\,y), & y\in [d_1,\,R],\\
\noalign{\vskip 1mm}
0, & y\in \mathbb{R}\setminus [d_1,\,R],
\end{cases}
\qquad\hbox{ for any $t\in [0,\,T)$,}
\end{equation}
and $\widetilde v(t,\,y)$ similarly.

 Next, let $M\ge M_0$ to be chosen below, where $M_0$ is given by \eqref{estimateTypeIb}.
For given $t_1\in (0,\,T)$, let
$(\overline{u},\,\overline{v})=(\overline{u}(t_1;\cdot,\cdot),$ $\,\overline{v}(t_1;\cdot,\cdot))$
be the solution of the following auxiliary problem:
\begin{equation}\label{defbaruv}
\left\{
  \begin{array}{ll}
    \overline{u}_{t}-\delta_1\overline{u}_{yy}=\widetilde u{\hskip 0.7pt}\widetilde v,
    &\,t_1<t<T,\ y \geq d_1,\\
    \overline{v}_{t}-\delta_2\overline{v}_{yy}=\widetilde u{\hskip 0.7pt}\widetilde v,
    &\,t_1<t<T,\ y \geq d_1,\\
     \overline{u}(t,\,d_1)=M(T-t)^{-1}, &t_1<t<T,\\
     \overline{v}(t,\,d_1)=M(T-t)^{-1}, &t_1<t<T,\\
   \overline{u}(t_1,\,y)=\widetilde u(t_1,\,y), &y  \geq d_1,\\
    \overline{v}(t_1,\,y)=\widetilde v(t_1,\,y), &y  \geq d_1.
\end{array}
\right.
\end{equation}
It is clear that $\overline u, \,\overline v \ge 0$ exist on $[t_1,\,T)\times [d_1,\,\infty)$.
 Also, using (\ref{estimateTypeIb}), (\ref{ineq1b}), (\ref{defutilde}) and $M\ge M_0$, we deduce from the maximum principle that
\begin{equation}\label{comptildebar}
\widetilde u\le \overline{u},\ \ \widetilde v\le \overline{v}
\quad\hbox{ on $[t_1,\,T)\times [d_1,\,\infty)$.}
\end{equation}

Further assuming $M\ge M_0^2$, hence $\widetilde u{\hskip 0.7pt}\widetilde v\le M(T-t)^{-2}$, we may use $M(T-t)^{-1}$
as a supersolution of the inhomogeneous, linear heat equations in (\ref{defbaruv}),
verified by $\overline u$ and $\overline v$ on $[t_1,\,T)\times [d_1,\,\infty)$,
and infer from the maximum principle that
\begin{equation}\label{estimbar2M1}
0\leq \overline u, \overline v\le M(T-t)^{-1} \quad\hbox{ on
$[t_1,\,T)\times [d_1,\,\infty).$}
\end{equation}

We next extend $(\overline{u},\,\overline{v})$ by odd reflection
for $y<d_1$, i.e., we set:
 \begin{eqnarray*}
 & &\overline u(t,\,d_1-y) = 2M(T-t)^{-1}-\overline u(t,\,d_1+y),\qquad t_1\le t<T,\ y>0,\\
 & &\overline v(t,\,d_1-y) = 2M(T-t)^{-1}-\overline v(t,\,d_1+y),\qquad t_1\le t<T,\ y>0.
  \end{eqnarray*}
 From (\ref{estimbar2M1}), we have
\begin{equation}\label{estimbar2M}
0\le \overline u, \overline v\le 2M(T-t)^{-1}
\quad\hbox{ on $[t_1,\,T)\times\mathbb{R}$}
\end{equation}
and (\ref{defutilde}), (\ref{comptildebar}) then guarantee
\begin{equation}\label{comptildebar2}
\widetilde u\le \overline u, \quad
\widetilde v\le \overline v
\quad\hbox{ on $[t_1,\,T)\times\mathbb{R}.$}
\end{equation}
 On the other hand, we see that the functions
$\overline u,\,  \overline v$ belong to $W^{1,2;k}_{loc}\bigl((t_1,\,T)\times \mathbb{R}\bigr)$ for all $1<k<\infty$
(actually their first order derivatives are continuous but $\overline u_{yy},\overline v_{yy}$ may have jumps at $y=d_1$, $y=R$ and $y=2d_1-R$). It is easy to check that $(\overline u,\,  \overline v)$ is a strong solution of
\begin{equation*}
\left\{
  \begin{array}{ll}
    \overline{u}_{t}-\delta_1\overline{u}_{yy}=F_1(t,\,y),
    &\,t_1<t<T,\ y\in\mathbb{R},\\
    \overline{v}_{t}-\delta_2\overline{v}_{yy}=F_1(t,\,y),
    &\,t_1<t<T,\ y\in\mathbb{R},\\
\end{array}
\right.
\end{equation*}
where
\begin{equation}\label{defF1}
F_1(t,\,y):=
\begin{cases}
2M(T-t)^{-2}-\widetilde u{\hskip 0.7pt}\widetilde v(t,\,2d_1-y), & y<d_1,\\
\noalign{\vskip 1mm}
\widetilde u{\hskip 0.7pt}\widetilde v(t,\,y), & y\ge d_1.
\end{cases}
\end{equation}

\smallskip
{\bf Step 2.} {\it Self-similar rescaling of  modified solutions.}
We now fix $d\in (d_1,\,d_0)$
(say, $d=(d_0+d_1)/2$) and pass to self-similar variables $(\sigma,\,\theta)$
around $(T,\,d)$, cf.~(\ref{defsimilvar}).
In these variables, we  first define the rescaled solution $(\widetilde w,\,\widetilde z)=(\widetilde w_d, \,\widetilde z_d)$,
associated with the extended solution $(\widetilde u,\,\widetilde v)$,
namely,
\begin{equation}\label{deftildewz}
\left\{
  \begin{array}{llll}
    \widetilde w(\sigma,\,\theta)&=&(T-t)\widetilde u(t,\,y),
    & \quad \hat\sigma\le \sigma<\infty,\ \theta\in \mathbb{R}, \\
    \widetilde z(\sigma,\,\theta)&=&(T-t)\widetilde v(t,\,y),
    & \quad \hat\sigma\le \sigma<\infty,\ \theta\in \mathbb{R}, \\
   \end{array}
\right.
\end{equation}
where $\hat\sigma=-\log T$.
For given $t_1\in (0,\,T)$, we also define
 $(\overline{w},\,\overline{z})=(\overline{w}_d(t_1;\cdot,\,\cdot),\,$ $\overline{z}_d(t_1;\cdot,\,\cdot))$,
 associated with the modified solution $(\overline{u}(t_1;\cdot,\,\cdot),\overline{v}(t_1;\cdot,\,\cdot))$
 (cf.~Step~1), given by
\begin{equation}\label{defbarwz}
\left\{
  \begin{array}{llll}
\overline w(\sigma,\,\theta)&=&(T-t)\overline u(t,\,y),
    & \quad \sigma_1\le \sigma<\infty,\ \theta\in \mathbb{R}, \\
\overline z(\sigma,\,\theta)&=&(T-t)\overline v(t,\,y),
    & \quad \sigma_1\le \sigma<\infty,\ \theta\in \mathbb{R}, \\
   \end{array}
\right.
\end{equation}
where $\sigma_1=-\log(T-t_1)> \hat \sigma$.

Set $\ell=d-d_1>0$. Owing to (\ref{estimbar2M}), (\ref{comptildebar2}), we have
 \begin{equation}\label{bound-2M}
 \widetilde w\le \overline w\le 2M, \quad
\widetilde z\le \overline z\le 2M
\quad\hbox{ on $[\sigma_1,\,\infty)\times \mathbb{R}$}
 \end{equation}
and, for all $\sigma\ge\hat\sigma$,
 \begin{equation}\label{monot-tilde}
 \theta\mapsto  \widetilde w(\sigma,\theta) \ \hbox{ and }\ \theta\mapsto  \widetilde z(\sigma,\theta)
 \ \hbox{ are nonincreasing for $\theta\in [-\ell e^{\sigma/2},\infty)$,}
  \end{equation}
due to (\ref{UrhoVrho}).
Then, using (\ref{eqsimilV}), (\ref{defF1}), we see that
$(\overline{w},\,\overline{z})$  is a strong solution of
\begin{equation}\label{13}
\hspace{-0,2cm}\left\{
  \begin{array}{llll}
 \overline{w}_\sigma-\mathcal{L}_{\delta_1}\overline{w}+\overline{w}&=&
 F_2(\sigma,\,\theta),
     & \quad \sigma_1  <  \sigma<\infty,\ \theta\in \mathbb{R}, \\
 \overline{z}_\sigma-\mathcal{L}_{\delta_2}\overline{z}+ \overline{z}&=& F_2(\sigma,\,\theta),
    & \quad \sigma_1 <\sigma<\infty,\ \theta\in \mathbb{R}, \\
  \end{array}
\right.
\end{equation}
where
\begin{multline}\label{defF2}
F_2(\sigma,\,\theta)=
e^{-2\sigma}F_1\bigl(T-e^{-\sigma},\,d+\theta e^{-\sigma/2}\bigr)
\le  \widetilde w(\sigma)\widetilde z(\sigma)+2M\chi_{\{\theta<-\ell e^{\sigma/2}\}}.
\end{multline}

In what follows, for any $\sigma\ge 0$,  we denote
$$\psi_\sigma:=2M \chi_{(-\infty,-\ell e^{\sigma/2})}.$$
Using the last two conditions in (\ref{defbaruv}),
 along with (\ref{deftildewz}), (\ref{defbarwz})
and (\ref{bound-2M}), we see that
\begin{equation}\label{split-barwz}
\overline w(\sigma_1)
\le\widetilde w(\sigma_1)+\psi_{\sigma_1}
\quad\hbox{and}\quad
\overline z(\sigma_1)
\le\widetilde z(\sigma_1)+\psi_{\sigma_1}.
\end{equation}

In the next steps, we shall estimate $(\widetilde w, \,\widetilde z)$
by using semigroup and delayed smoothing arguments.

\smallskip

{\bf Step 3.} {\it New auxiliary functions and first semigroup estimates}.
Let us consider the semigroup majorant:
 $$(S(\sigma))_{\sigma\geq0}=(T_{\delta_1}(\sigma)+T_{\delta_2}(\sigma))_{\sigma\geq0}.$$
 For given $t_1\in (0,\,T)$, we set again $\sigma_1=-\log(T-t_1)$ and consider
$(\overline{w},\,\overline{z})=\bigl(\overline{w}_d(t_1;\cdot,\,\cdot),\,$ $\overline{z}_d(t_1;\cdot,\,\cdot)\bigr)$,
defined in (\ref{defbarwz}).
 We use (\ref{13}) and the variation of constants formula to write
$$
\overline{w}(\sigma_1+\sigma)=e^{- \sigma}T_{\delta_1}(\sigma)\overline w(\sigma_1)
 +\,\int_0^\sigma
e^{-(\sigma-\tau)}T_{\delta_1}(\sigma-\tau)F_2(\sigma_1+\tau,\,\cdot)d\tau
$$
for all $\sigma>0$, hence, by (\ref{defF2}),
\begin{multline}
\overline{w}(\sigma_1+\sigma)\leq
   e^{- \sigma}S(\sigma)\overline{w}(\sigma_1)
   +\displaystyle \int_0^\sigma e^{-(\sigma-\tau)}S(\sigma-\tau)
\psi_{\sigma_1+\tau}d\tau\\
+\displaystyle \int_0^\sigma
e^{-(\sigma-\tau)}S(\sigma-\tau)\left(\widetilde{w}(\sigma_1+\tau)\widetilde z(\sigma_1+\tau)\right)d\tau.
\quad\label{estim-wbar}
  \end{multline}
Similarly, by exchanging the roles of
$\overline{w}$, $\widetilde{w}$ and $\overline{z},$ $\widetilde{z},$
we obtain
\begin{multline}
\overline{z}(\sigma_1+\sigma)\leq
   e^{- \sigma}S(\sigma)\overline{z}(\sigma_1)
   +\displaystyle \int_0^\sigma e^{-(\sigma-\tau)}S(\sigma-\tau)
\psi_{\sigma_1+\tau}d\tau\\
+\displaystyle \int_0^\sigma
e^{-(\sigma-\tau)}S(\sigma-\tau)\left(\widetilde{w}(\sigma_1+\tau)\widetilde z(\sigma_1+\tau)\right)d\tau.
\quad\label{estim-zbar}
  \end{multline}

Let us next introduce the auxiliary functions $(\hat{w},\,\hat{z})=\bigl(\hat{w}_d(t_1;\cdot,\,\cdot),\,\hat{z}_d(t_1;\cdot,\,\cdot)\bigr)$,
given~by:
\begin{multline}
\hat{w}(\sigma_1+\sigma):=   e^{- \sigma}S(\sigma)
   \Bigl[\widetilde w(\sigma_1)+\psi_{\sigma_1}\Bigr]+\displaystyle \int_0^\sigma
e^{-(\sigma-\tau)}S(\sigma-\tau)\psi_{\sigma_1+\tau}d\tau\\
+\displaystyle\int_0^\sigma
e^{-(\sigma-\tau)}S(\sigma-\tau)\left(\widetilde{w}(\sigma_1+\tau)\widetilde z(\sigma_1+\tau)\right)d\tau
\label{def-w-hat}
 \end{multline}
  and
\begin{multline}
\hat{z}(\sigma_1+\sigma):=   e^{- \sigma}S(\sigma)
   \Bigl[\widetilde z(\sigma_1)+\psi_{\sigma_1}\Bigr]+\displaystyle \int_0^\sigma
e^{-(\sigma-\tau)}S(\sigma-\tau)\psi_{\sigma_1+\tau}d\tau\\
+\displaystyle\int_0^\sigma
e^{-(\sigma-\tau)}S(\sigma-\tau)\bigl(\widetilde{w}(\sigma_1+\tau)\widetilde z(\sigma_1+\tau)\bigr)d\tau.
\label{def-z-hat}
 \end{multline}
 Then it follows from  (\ref{estim-wbar}),  (\ref{estim-zbar}),
(\ref{bound-2M}) and (\ref{split-barwz}), that
\begin{equation}\label{comptildehat}
\widetilde w\leq \hat{w}\quad \mbox{ and }\quad \widetilde z\leq \hat{z}.
\end{equation}
 Also, by the semigroup property, we have
\begin{multline}
\label{prop-w-hat-semi}
\hat{w}(\sigma_2+\sigma)=  e^{- \sigma}S(\sigma)\hat{w}(\sigma_2)
+\displaystyle \int_0^\sigma
e^{-(\sigma-\tau)}S(\sigma-\tau)\psi_{\sigma_2+\tau}d\tau\\
+\displaystyle\int_0^\sigma
e^{-(\sigma-\tau)}S(\sigma-\tau)\left(\widetilde{w}(\sigma_2+\tau)\widetilde z(\sigma_2+\tau)\right)d\tau,
\quad \sigma_2\ge \sigma_1,\ \sigma\ge 0.
 \end{multline}
In particular, using (\ref{def-w-hat}), (\ref{comptildehat}), (\ref{bound-2M}) and dropping the exponential factors,
we have
$$
\hat{w}(\sigma_2+\sigma)\le S(\sigma)\hat{w}(\sigma_2)
+\displaystyle \int_0^\sigma
S(\sigma-\tau)\psi_{\sigma_2+\tau}d\tau\\
+2M\displaystyle\int_0^\sigma
S(\sigma-\tau)\widetilde w(\sigma_2+\tau)d\tau.
$$
 By a standard argument (see e.g. \cite[pp. 418-419]{MST}), it follows that
\begin{equation} \label{prop-w-hat}
\hat w(\sigma_2+\sigma)\leq e^{2M\sigma}S(\sigma)\hat w(\sigma_2)
+ \displaystyle\int_0^\sigma e^{2M(\sigma-\tau)}S(\sigma-\tau) \psi_{\sigma_2+\tau} \,d\tau,
\quad \sigma_2\ge \sigma_1,\ \sigma\ge 0.
 \end{equation}

On the other hand, subtracting  (\ref{def-w-hat}) and (\ref{def-z-hat}) (with $\sigma$ replaced by
$\sigma+\sigma_2-\sigma_1$), we have
\begin{equation}\label{comphathat}
\hat{z}(\sigma_2+\sigma)= \hat{w}(\sigma_2+\sigma)+ e^{- \sigma+\sigma_1-\sigma_2}S(\sigma+\sigma_2-\sigma_1)\Bigl[\widetilde z(\sigma_1)-\widetilde w(\sigma_1)\Bigr],\quad \sigma_2\ge\sigma_1,\, \sigma\ge 0
\end{equation}
and we deduce from (\ref{comptildehat}) and (\ref{prop-w-hat-semi}) that
\begin{multline}
\hat{w}(\sigma_2+\sigma)\leq
e^{- \sigma}S(\sigma)\hat w(\sigma_2)
+\displaystyle \int_0^\sigma e^{-(\sigma-\tau)}S(\sigma-\tau)\psi_{\sigma_2+\tau}d\tau
+\displaystyle\int_0^\sigma e^{-(\sigma-\tau)}S(\sigma-\tau)\hat{w}^2(\sigma_2+\tau)d\tau \\
+\displaystyle\int_0^\sigma e^{-(\sigma-\tau)}S(\sigma-\tau)\Bigl[e^{- \tau}\bigl(S(\tau+\sigma_2-\sigma_1)\widetilde z(\sigma_1)\bigr)
\hat w(\sigma_2+\tau)\Bigr]d\tau. \notag
\end{multline}
Note that, by (\ref{11a}),
$$\|S(s)\phi\|_\infty\le 2\|\phi\|_\infty\quad\hbox{ for all $\phi\in L^\infty(\mathbb{R})$ and all $s\ge 0$ }.$$
Since $\|\widetilde z(\sigma_1)\|_\infty\le 2M$ by (\ref{bound-2M}), we then obtain
\begin{multline}
\hat{w}(\sigma_2+\sigma)\leq
e^{- \sigma}S(\sigma)\hat w(\sigma_2)
+\displaystyle \int_0^\sigma e^{-(\sigma-\tau)}S(\sigma-\tau)\psi_{\sigma_2+\tau}d\tau\\
+\displaystyle\int_0^\sigma e^{-(\sigma-\tau)}S(\sigma-\tau)\hat{w}^2(\sigma_2+\tau)d\tau
+4M e^{-\sigma}\displaystyle\int_0^\sigma S(\sigma-\tau)\hat w(\sigma_2+\tau) d\tau.
\label{wchapeaunew}
\end{multline}

\smallskip
{\bf Step  4.} {\it Small time estimate of rescaled solutions.} At this point, we set
$\bar\delta=\max(\delta_1,\delta_2)$ and $K=K_{\bar\delta}$,
and let $\sigma^\ast$ be given by Lemma~\ref{66}(2), with $k=1$ and $m=2$. We have
 \begin{equation}\label{57}
\|S(\sigma)\phi\|_{L_{K}^2}\leq
\widetilde C_0 \|\phi\|_{L_{K}^1},\quad \sigma\geq \sigma^\ast,\,\phi\in L^\infty(\mathbb{R}),
\end{equation}
as well as
 \begin{equation}
 \label{51}
\|S(\sigma)\phi\|_{L_{K}^k}\leq
\widetilde C\|\phi\|_{L_{K}^k},\quad \sigma\geq 0,\,\phi\in
L^\infty(\mathbb{R}), \, k\in [1,\infty),
 \end{equation}
with $\widetilde C_0=\widetilde C_0(\delta)\ge 1$ and $\widetilde C=\widetilde C(\delta)\ge 1$.
Also, as in \cite{MST}, we recall that
$$
\|S(\sigma)\chi_{\{\theta<-A\}}\|_{L_{K}^k}\le
\widetilde C\left((4\pi\overline\delta)^{-1/2}\int_{-\infty}^{-A} \exp\Bigl({\frac{-\theta^2}{4\overline\delta}}\Bigr)\,d\theta\right)^{1/k}
\le C_0\exp(-(8k\overline\delta)^{-1}A^2),
$$
for all $\sigma\ge 0$, $A>0$, and $k\in [1,\infty)$,
with $C_0=C_0(\delta)\ge 1$. In particular
\begin{equation}\label{estim-indicator}
\|S(\sigma)\psi_{\tau}\|_{L_{K}^k}
\le 2MC_0\exp(-(8k\overline\delta)^{-1}\ell^2 e^\tau),
\quad\hbox{ for all $\sigma,\tau\ge 0$ and $k\in [1,\infty)$.}
\end{equation}

Let now $\eta>0$. We claim that there exists $\tau_1\in(0,\,T)$, depending only on $\eta$ and and on the parameters
\begin{equation}\label{listparam}
\delta_1,\,\delta_2,\,d_0,\,d_1,\,M,\,R,\,T,
\end{equation}
with the following property:
    \begin{eqnarray}
&\hbox{ For any $t_1\in[T-\tau_1,\,T)$ such that $(T-t_1) U(t_1,\,d_1)\leq \eta,\ \ $}
\notag\\
&\hbox{ we have  $\|\hat{w}(\sigma_1+\sigma)\|_{L_{K}^1}\leq \widetilde C_1 \eta$\,
for all $0<\sigma\leq\sigma^\ast$,}
   \label{21}
    \end{eqnarray}
 where $\hat{w}=\hat{w}(t_1;\cdot,\cdot)$ is given by (\ref{def-w-hat}), $\sigma_1=-\log(T-t_1)$ and $\widetilde C_1=2\widetilde Ce^{2M\sigma^{\ast}}>0$.

First observe that, by the assumption $(T-t_1) U(t_1,\,d_1)\leq \eta$ and owing to (\ref{UrhoVrho}),
we have $\widetilde w(\sigma_1,\, \cdot)\le\eta$
 on $\mathbb{R}$, hence
\begin{equation}\label{cond2eta}
\|\widetilde{w}(\sigma_1)\|_{L_{K}^1}
\le \eta.
\end{equation}
We apply (\ref{prop-w-hat}) with  $\sigma_2=\sigma_1$.
Using (\ref{51}), (\ref{estim-indicator}), (\ref{cond2eta}), $\hat{w}(\sigma_1)=\widetilde w(\sigma_1)+\psi_{\sigma_1}$,
$e^{\sigma_1}=(T-t_1)^{-1}\ge \tau_1^{-1}$ and assuming $\tau_1<1$,
we deduce that, for $0\le \sigma\le\sigma^\ast$,
\begin{eqnarray*} 
 \|\hat{w}(\sigma_1+\sigma)\|_{L_{K}^1}
 &\leq& e^{2M\sigma^*} \left(\|S(\sigma) \widetilde w(\sigma_1)\|_{L_{K}^1}+\|S(\sigma)\psi_{\sigma_1}\|_{L_{K}^1}
+\displaystyle \int_0^\sigma \|S(\sigma-\tau) \psi_{\sigma_1+\tau}\|_{L_{K}^1}d\tau\right) \\
& \leq& e^{2M\sigma^*} \left(\widetilde C\eta+2M(1+\sigma^*)C_0\exp(-(8\overline\delta\tau_1)^{-1}\ell^2)\right).
 \end{eqnarray*}
For $\tau_1\in(0,\,T)$ sufficiently small, depending only on $\eta$ and on the parameters
in~(\ref{listparam}), we thus get (\ref{21}) with
$\widetilde C_1=2\widetilde Ce^{2M\sigma^{\ast}}$.

    \smallskip
{\bf Step 5.} {\it Large time estimate of rescaled solutions.}
We claim that there exist $\eta>0$ and $\tau_0\in(0,\,\tau_1(\eta)]$,
depending only on the parameters in~(\ref{listparam}), with the following property:
     \begin{eqnarray*}
&\hbox{ for any $t_1\in[T-\tau_0,\,T)$  such that $(T-t_1) U(t_1,\,d_1)\leq \eta$,} \notag \\
&\hbox{ we have $\mathcal{A}_{\eta,\,t_1}=(0,\,\infty)$,}
    \end{eqnarray*}
where
$$\mathcal{A}_{\eta,t_1}=\Bigl\{\sigma>0\,;\
\|\hat{w}(\sigma_1+\sigma^\ast+\tau)\|_{L_{K}^1}\leq
\widetilde C_2\eta e^{-\tau},\ \ \tau\in[0,\,\sigma]\Bigr\},$$
 where $\hat{w}=\hat{w}(t_1;\cdot,\cdot)$ is given by (\ref{def-w-hat}), $\sigma_1=-\log(T-t_1)$ and $\widetilde C_2=3\widetilde C\widetilde C_1 e^{4M\widetilde C}$.

Observe that $\mathcal{A}_{\eta,\,t_1}\neq\emptyset$, due to
(\ref{21}) and the continuity of the function $\sigma\mapsto
e^{\sigma}\|\hat{w}(\sigma_1+\sigma^\ast+\sigma)\|_{L_{K}^1}$.
We denote
$$\overline T=\sup \mathcal{A}_{\eta,\,t_1} \in (0,\,\infty].$$
Assume for contradiction that $\overline T<\infty$.
Then, taking (\ref{21}) into account, we have
\begin{eqnarray}\label{22}
  \hspace{-2cm}& &\|\hat{w}(\sigma_1+\sigma^\ast+\sigma)\|_{L_{K}^1}
  \leq \widetilde C_2\eta e^{- \sigma},\quad -\sigma^\ast\leq \sigma\leq \overline T.
  \end{eqnarray}

With help of the delayed regularizing effect of the semigroup $(S(\sigma))$, we first establish the following $L^2_K$
decay estimate:
 \begin{eqnarray}\label{23}
       & &\|\hat{w}(\sigma_1+\sigma^\ast+\tau)\|_{L_{K}^2} \leq
       2\widetilde C_0\widetilde C_2 e^{(2M+1)\sigma^\ast} \eta e^{-\tau},\ \ 0\le\tau\leq \overline T.
     \end{eqnarray}
To do so, for $0\leq\tau\leq \overline T$,
we apply (\ref{prop-w-hat}) with $\sigma_2=\sigma_1+\tau$ and $\sigma=\sigma^\ast$.
Using (\ref{57}), (\ref{51}),  (\ref{estim-indicator}), (\ref{cond2eta}),
$e^{\sigma_1}=(T-t_1)^{-1}\ge \tau_0^{-1}$ and assuming $\tau_0<1$, we obtain
\begin{multline*}
\|\hat{w}(\sigma_1+\tau+\sigma^*)\|_{L_{K}^2}
\leq e^{2M \sigma^*}\|S(\sigma^*) \hat w(\sigma_1+\tau)\|_{L_{K}^2}
+\displaystyle \int_0^{\sigma^*}e^{2M(\sigma^*-s)}\|S(\sigma^*-s)\psi_{\sigma_1+\tau+s}\|_{L_{K}^2}ds \\
\leq e^{2M \sigma^*}
 \left\{\widetilde C_0\|  \widetilde w(\sigma_1+\tau)\|_{L_{K}^1}+ 2MC_0\sigma^*\exp(-(16\overline\delta\tau_0)^{-1}\ell^2e^\tau)
\right\}
\end{multline*}
hence, by (\ref{22}),
$$
\|\hat{w}(\sigma_1+\tau+\sigma^*)\|_{L_{K}^2}
\leq e^{(2M+1) \sigma^*}
\left\{\widetilde C_0\widetilde C_2\eta e^{- \tau}
+2MC_0\sigma^*\exp(-(16\overline\delta\tau_0)^{-1}\ell^2e^\tau)\right\}.
$$
For $\tau_0\in (0,\tau_1(\eta)]$ sufficiently small, depending only on $\eta$ and on the parameters in~(\ref{listparam}),
we deduce (\ref{23}).

Next, using 
the $L^2_K$ decay estimate (\ref{23}), we shall derive from the semigroup inequality (\ref{wchapeaunew})
an $L^1_K$ decay estimate which leads to a contradiction with the definition of $\overline T$.
The important features of (\ref{wchapeaunew}) used here are the quadraticity of the third term and
the exponential factor in the last term of the RHS.
Thus for any $0<\sigma\leq \overline T$, applying (\ref{wchapeaunew}) with $\sigma_2=\sigma_1+\sigma^*$, we have
\begin{multline}
\|\hat{w}(\sigma_1+\sigma^*+\sigma)\|_{L_{K}^1}\leq e^{- \sigma}\|S(\sigma)
   \hat w(\sigma_1+\sigma^*)\|_{L_{K}^1}
   +\displaystyle \int_0^{\sigma} e^{-(\sigma-s)}\|S(\sigma-s)\psi_{\sigma_1+\sigma^*+s}\|_{L_{K}^1}ds\\
+\displaystyle\int_0^{\sigma}
e^{-(\sigma-s)}\|S(\sigma-s)\hat{w}^2(\sigma_1+\sigma^*+s)\|_{L_{K}^1}ds
+4Me^{-\sigma}\displaystyle\int_0^{\sigma}
\|S(\sigma-s)\hat w(\sigma_1+\sigma^*+s)\|_{L_{K}^1}ds. \notag
\end{multline}
Letting
$$H(\sigma):=e^{\sigma}\|\hat{w}(\sigma_1+\sigma^*+\sigma)\|_{L_{K}^1}$$
and using (\ref{51}), (\ref{estim-indicator}) and $e^{\sigma_1}\ge \tau_0^{-1}$, it follows that
\begin{multline}
H(\sigma)\leq \widetilde C\|\hat w(\sigma_1+\sigma^*)\|_{L_{K}^1}
+2MC_0 \displaystyle \int_0^{\sigma}
e^{s}\exp\bigl(-(8\overline\delta\tau_0)^{-1}\ell^2e^{s}\bigr)\,ds\\
+\widetilde C\displaystyle\int_0^{\sigma}
e^{s}\|\hat{w}(\sigma_1+\sigma^*+s)\|^2_{L_{K}^2}ds+4M\widetilde C\displaystyle\int_0^{\sigma}e^{-s}H(s)ds.
 \notag
\end{multline}
By taking $\tau_0$ possibly smaller (dependence as above), we may ensure that
$$2M C_0 \displaystyle\int_0^\infty e^{s}
\exp\bigl(-(8\overline\delta\tau_0)^{-1}\ell^2e^{s}\bigr)\,
ds\le \eta^2.$$
Using (\ref{21}) and (\ref{23}), it then follows that
$$
H(\sigma)
\le \widetilde C\widetilde C_1\eta+\eta^2
+4\widetilde C \bigl(\widetilde C_0\widetilde C_2 e^{(2M+1)\sigma^\ast}\bigr)^2
\eta^2\displaystyle\int_0^{\sigma}e^{-s}ds
+4M\widetilde C\displaystyle\int_0^{\sigma}e^{-s}H(s)ds.
$$
Choosing $\eta>0$ sufficiently small (depending only on $\delta_1, \delta_2, M$), we obtain
$$
H(\sigma)\leq 2\widetilde C\widetilde C_1 \eta+4M\widetilde C\displaystyle\int_0^{\sigma}e^{-s}H(s)ds
$$
hence, by Gronwall's lemma,
$$H(\sigma)\leq 2\widetilde C\widetilde C_1 \eta \exp\left\{4M\widetilde C\int_0^\sigma e^{-s}ds\right\}\leq 2\widetilde C\widetilde C_1 e^{4M\widetilde C}\eta,\qquad 0<\sigma\leq \overline T.$$
Finally, for $\sigma=\overline T$, by definition of $\overline T$, we obtain
$$3\widetilde C\widetilde C_1 e^{4M\widetilde C}\eta=H(\overline T)\leq 2\widetilde C\widetilde C_1 e^{4M\widetilde C}\eta$$
which is a contradiction.
Consequently, $\overline T=\infty$
 and the claim is proved.

    \smallskip
{\bf Step 6.} {\it Conclusion.} Let $\eta,\,\tau_0$ be as in Step~5
and let $t_1\in[T-\tau_0,\,T)$ satisfy
$$(T-t_1) U(t_1,\,d_1)\leq \eta.$$
It follows from  the definition of $\mathcal{A}_{\eta,\,t_1}$ that
$$
\sup_{\sigma\ge \sigma_1+\sigma^\ast}
\Bigl(e^{\sigma}\|\hat{w}(\sigma)\|_{L_{K}^1}\Bigr)<\infty.
$$
By (\ref{comphathat}) and (\ref{51}), it follows that
$$
\sup_{\sigma\ge \sigma_1+\sigma^\ast}
\Bigl(e^{\sigma}\|\hat{z}(\sigma)\|_{L_{K}^1}\Bigr)<\infty.
$$
Consequently, by (\ref{comptildehat}),
\begin{equation}\label{Lambda0tris}
\Lambda_0=\sup_{\sigma\ge \sigma_1+\sigma^\ast}
\Bigl(e^{\sigma}\|\widetilde{w}(\sigma)\|_{L_{K}^1}+e^{\sigma}\|\widetilde{z}(\sigma)\|_{L_{K}^1}\Bigr)<\infty.
\end{equation}
Set $L:=\int_{-1}^0 K(\theta)\,d\theta>0$.  For all $t\in[T-\ell^{-2},\,T)$,
recalling (\ref{defsimilvar}), we have $\ell e^{\sigma/2}\ge 1$, hence
\begin{equation}\label{CompTildeMonot}
\widetilde{w}(\sigma,\,0)\le L^{-1}\int_{-1}^0\widetilde{w}(\sigma,\,\theta)K(\theta)\,d\theta,
\qquad
\widetilde{z}(\sigma,\,0)\le L^{-1}\int_{-1}^0\widetilde{z}(\sigma,\,\theta)K(\theta)\,d\theta,
\end{equation}
owing to (\ref{monot-tilde}).
Let then $\hat t_1=T-\min\bigl(\ell^{-2},e^{-(\sigma_1+\sigma^\ast)}\bigr)$.
It follows from (\ref{defutilde}), (\ref{deftildewz}), (\ref{Lambda0tris}), (\ref{CompTildeMonot}) that,
for all $t\in[\hat t_1,\,T)$,
\begin{eqnarray*}
u(t,\,d)+v(t,\,d)
&=&e^{ \sigma}\widetilde{w}(\sigma,\,0)+e^{\sigma}\widetilde{z}(\sigma,\,0) \\
&\leq& 2L^{-1}\Bigl(e^{\sigma}\|\widetilde{w}(\sigma)\|_{L_{K}^1}
+e^{\sigma}\|\widetilde{z}(\sigma)\|_{L_{K}^1}\Bigr) \le 2L^{-1}\Lambda_0.
\end{eqnarray*}
Using (\ref{UrhoVrho}), we conclude that $d_0>d$ is not a blow-up point.

We reach the same conclusion if we assume instead that $(T-t_1) V(t_1,\,d_1)\leq \eta$
(exchange the roles of $U$ and $V$). The proposition is proved.
\qed

\section{Proof of Theorem \ref{th1}}
\setcounter{equation}{0}
In this section we prove Theorem \ref{th1}. Proposition~\ref{15} immediately yields the following
asymptotic comparison properties near possible nonzero blow-up points.

\begin{lem}\label{30}
Under the hypotheses of Theorem~\ref{th1}, assume that there exists $\rho_0\in(0,\,R)$ such that
$$\underset{t\rightarrow
T}{\limsup}\,\big{(}u(t,\,\rho_0)+v(t,\,\rho_0)\big{)}=\infty.$$
Then, for any compact subinterval $[\rho_1,\,\rho_2]\subset (0,\,\rho_0)$, there exist
real constants $C_1,\,C_2$  (possibly depending on the solution $(u,v)$ and
on $\rho_0, \rho_1, \rho_2$),
such that
\begin{equation}\label{310}
C_1\le \log(T-t)+ qu(t,\,\rho)\le C_2\quad\mbox{on}\quad[T/2,\,T)\times[\rho_1,\,\rho_2]
\end{equation}
and
\begin{equation}\label{311}
C_1\le \log(T-t)+ pv(t,\,\rho)\le C_2\quad\mbox{on}\quad[T/2,\,T)\times[\rho_1,\,\rho_2].
\end{equation}
In particular, there exist constants $C_1', C_2'>0$ such that
\begin{equation}\label{31}
C'_1\leq\frac{e^{qu}(t,\,\rho)}{e^{pv}(t,\,\rho)}\leq C'_2\quad\mbox{on}\quad[T/2,\,T)\times[\rho_1,\,\rho_2].
\end{equation}
\end{lem}

To prove Theorem \ref{th1}, we introduce the auxiliary functions $J,\, \overline{J}$ defined by
\begin{equation}\label{47}
J(t,\,\rho)=u_\rho+\eps c(\rho)e^{\gamma u},\quad
\overline{J}(t,\,\rho)=v_\rho+\eps c(\rho)e^{\overline\gamma v},
\end{equation}
with
\begin{equation}\label{48}
c(\rho)=\sin^2\left(
\frac{\pi(\rho-\rho_1)}{\rho_2-\rho_1}\right),\quad \rho_1\leq \rho\leq \rho_2,
\end{equation}
where $\gamma, \overline\gamma>0$,  $\eps>0$ and
$\rho_2>\rho_1>0$ are to be fixed. We note that
 $J$, $\overline{J}$ $\in C^{1,2}((0,\,T)\times[0,\,R])$
 by parabolic regularity.

\begin{lem}\label{49}
 Under the hypotheses of Theorem  \ref{th1},
assume that there exists $\rho_0\in(0,\,R)$ such that
$$\underset{t\rightarrow
T}{\limsup}\,\big{(}u(t,\,\rho_0)+v(t,\,\rho_0)\big{)}=\infty$$
and let $\rho_1=\rho_0/4$ and $\rho_2=\rho_0/2$.
Then there exist $\gamma, \overline\gamma\in(0,1)$ and $T_1 \in(0,\,T)$, such that,
for any $\eps\in(0,\,1]$,
 the functions $J$ and $\overline{J}$ defined in (\ref{47})--(\ref{48})
satisfy
\begin{equation}\label{50}
    \hspace{-0,5cm}\left\{
      \begin{array}{lll}
           \hspace{-0,2cm}\hfill J_t-\delta_1 J_{\rho\rho}-\delta_1\displaystyle\frac{n-1}{\rho}J_\rho+\delta_1\frac{n-1}{\rho^2}J&
           \leq pe^{pv}\overline{J}-2\eps\delta_1\gamma c'e^{\gamma u}J, & \hbox{ }\\
      \hspace{-0,2cm}\hfill \overline{J}_t-\delta_2\overline{J}_{\rho\rho}-\delta_2\displaystyle\frac{n-1}{\rho}\overline{J}_\rho+
      \delta_2\frac{n-1}{\rho^2}\overline{J}
        &\leq qe^{qu}J-2\eps \delta_2 \gamma \, c'e^{\overline\gamma v}\overline{J}, &\hbox{ }
      \end{array}
    \right.
\end{equation} for all $(t,x)\in[T_1,\,T)\times(\rho_1,\, \rho_2 ).$
\end{lem}

The proof of the previous lemma is a simple modification of that in \cite[Lemma 2.1]{MST}. We give it for completeness.
\begin{proof} Let $H= e^{\gamma u}$. By differentiation of (\ref{47}), we have
\begin{eqnarray*}
 J_t-\delta_1 J_{\rho\rho}\hspace{-0,6cm}& &=(u_\rho)_t +\eps c H_t-\delta_1(u_{\rho\rho})_\rho-\delta_1\eps c''H-2\delta_1\eps c' H_\rho-\delta_1\eps c H_{\rho\rho}\\
 \hspace{-0,6cm}& &=(u_t-\delta_1 u_{\rho\rho})_\rho+\eps\Big{(}c\big{(}H_t-\delta_1 H_{\rho\rho}\big{)}-2\delta_1 c'H_\rho-\delta_1 c''H\Big{)}.
\end{eqnarray*}
  Set $m=0$ (resp., $m=1$) if $f$ is given by \eqref{efgm=0} (resp., \eqref{efgm=1}).
 By the first equation in (\ref{e1}), we get
\begin{eqnarray*}
 (u_t-\delta_1 u_{\rho\rho})_\rho=\left(\delta_1\frac{n-1}{\rho}u_\rho+e^{pv}-m\right)_\rho=\delta_1\frac{n-1}{\rho}u_{\rho\rho}-
 \delta_1\frac{n-1}{\rho^2}u_\rho+pe^{pv} v_\rho
\end{eqnarray*}
 and
\begin{eqnarray*}
H_t-\delta_1 H_{\rho\rho}\hspace{-0,6cm}& &=\gamma e^{\gamma u}u_t-\delta_1\gamma^2e^{\gamma u} u^2 _\rho-\delta_1\gamma e^{\gamma u}u_{\rho\rho}\\
\hspace{-0,6cm}& &\leq\gamma e^{\gamma u}\big{(}u_t-\delta_1 u_{\rho\rho}\big{)}
 =\gamma e^{\gamma u}\left(\delta_1\frac{n-1}{\rho}u_\rho +\,e^{pv}-m\right).
\end{eqnarray*}
 Using this, along with
$u_\rho=J-\eps c e^{\gamma u}$ and
$v_\rho=\overline{J}-\eps  c e^{\overline\gamma v}$, we obtain
\begin{eqnarray*}
J_t-\delta_1 J_{\rho\rho}\hspace{-0,6cm}& &\leq\delta_1\frac{n-1}{\rho}\left(J-\eps ce^{\gamma u}\right)_\rho-
\delta_1\frac{n-1}{\rho^2}\left(J-\eps ce^{\gamma u}\right)+pe^{pv}\left(\overline{J}-\eps  c e^{\overline\gamma v}\right)\\
 \hspace{-0,6cm}& &\hspace{0,3cm}+\eps e^{\gamma u}\left[\gamma c\left(\delta_1\frac{n-1}{\rho}u_\rho +e^{pv}-m \right)-
 2 \gamma\delta_1 c'u_\rho-\delta_1 c''\right]\\
 \hspace{-0,6cm}& &=\delta_1\frac{n-1}{\rho}J_\rho-\delta_1\eps\frac{n-1}{\rho}c'e^{\gamma u}-\delta_1\eps c\frac{n-1}{\rho}\gamma
 e^{\gamma u}u_\rho-\delta_1\frac{n-1}{\rho^2}J\\
 \hspace{-0,6cm}& &\hspace{0,3cm}+\,\delta_1 \eps \frac{n-1}{\rho^2}c e^{\gamma u} +
 pe^{pv} \left(\overline{J}-\eps  c e^{\overline\gamma v}\right)\\
 \hspace{-0,6cm}& &\hspace{0,3cm}+\,\eps e^{\gamma u}\left[\gamma c\left(\delta_1\frac{n-1}{\rho}
 u_\rho+e^{pv}-m\right)-2 \delta_1\gamma c'\left(J-\eps c e^{\gamma u}\right)-\delta_1 c''\right].
\end{eqnarray*}
Consequently,
\begin{equation}
 J_t-\delta_1 J_{\rho\rho}-\delta_1\frac{n-1}{\rho}J_\rho+\delta_1\frac{n-1}{\rho^2}J\leq pe^{pv} \overline{J}
 -2\eps\delta_1\gamma c'e^{\gamma u}J+\eps H_1,\quad \label{52}
\end{equation}
with
$$
 H_1:= - pc e^{\overline\gamma v}e^{pv}
 +\,e^{\gamma u}\left[\gamma c(e^{pv}-m) +2\delta_1\eps\gamma c'c e^{\gamma u}+\delta_1 \left(\frac{n-1}{\rho}\Big{(}\frac{c}{\rho}-c'\Big{)}-c''\right)\right].
$$
 For convenience, we set
$$\xi(\rho)=\frac{n-1}{\rho}\Big{(}\frac{1}{\rho}-\frac{c'}{c}\Big{)}-\frac{c''}{c},
\qquad \rho_1<\rho<\rho_2$$
and, on $(0,T)\times (\rho_1,\rho_2)$,
\begin{equation}
\widetilde{H}_1:=\frac{H_{1}}{c\,e^{\gamma u}}=-\frac{e^{\overline\gamma v}}{e^{\gamma u}}pe^{pv} \\
+\gamma  (e^{pv}-m)+2\delta_1\,\eps \gamma c'e^{\gamma u}+ \delta_1\xi(\rho).
\label{53}
\end{equation}
 Note that, up to now, our calculations made use of (\ref{e1})
through the first PDE only.
Thus, by replacing $\delta_1$ with $\delta_2$ and exchanging the roles of
$u,\, v,\; p\; \gamma,\; \overline{\gamma}$ 
and $v,\; u,\; q,\; \overline{\gamma},\; \gamma$, 
we get
\begin{multline}
\overline{J}_t-\delta_2\overline{J}_{\rho\rho}-\delta_2\frac{n-1}{\rho}\overline{J}_\rho+\delta_2\frac{n-1}{\rho^2}\overline{J}\leq
qe^{qu} J-2\eps \delta_2\overline\gamma\, c'e^{\overline\gamma v}\overline{J}+\eps H_2,
\end{multline}
with
\begin{equation}
\widetilde{H}_{2}:=\frac{H_{2}}{ c\,e^{\overline\gamma v}}:=-\frac{e^{\gamma u}}{e^{\overline\gamma v}}qe^{qu}
+\overline\gamma(e^{qu}-m)+2 \delta_2\,\eps\overline\gamma \, c'e^{\overline\gamma v}
+\delta_2\xi(\rho).
\label{55}
\end{equation}

Next setting $\ell=\rho_2-\rho_1=\rho_0/4$, we have
$$-\frac{c'}{c}
 =-\frac{2\pi}{\ell}\cot\Bigl(\frac{\pi(\rho-\rho_1)}{\ell}\Bigr)
\quad\hbox{ and }\quad
 -\frac{c''}{c}=-\frac{2\pi^{2}}{\ell^2}\cot^2\Bigl(\frac{\pi(\rho-\rho_1)}{\ell}\Bigr)+\frac{2\pi^2}{\ell^2}\\
$$
 hence,
\begin{eqnarray*}
\xi(\rho)
=\frac{n-1}{\rho^2}+\frac{2\pi^2}{\ell^2}
- \frac{2\pi}{\ell}\left[\frac{n-1}{\rho}+\frac{\pi}{\ell}
\cot\Bigl(\frac{\pi(\rho-\rho_1)}{\ell}\Bigr)\right]\cot\Bigl(\frac{\pi(\rho-\rho_1)}{\ell}\Bigr).
\end{eqnarray*}
 It follows that
\begin{equation*}
 \xi(\rho)
\ \underset{\rho\rightarrow \rho_1^+}{\longrightarrow} -\infty
\quad\hbox{ and }\quad
\xi(\rho)
\ \underset{\rho\rightarrow \rho_2^-}{\longrightarrow} -\infty.
\end{equation*}
Since  $\xi$
 is continuous on $(\rho_1,\,\rho_2)$, then there exists $C_3=C_3(n,\,\rho_0)>0$ such that
\begin{eqnarray}\label{58}
\xi(\rho)\leq
C_3,\quad\hbox{for all}\,\,\rho\in(\rho_1,\,\rho_2).
\end{eqnarray}
By (\ref{53}), (\ref{55}) and (\ref{58}), we obtain,
for some $C_4=C_4(\delta_1,\, \delta_2,\,\rho_0)>0$,

\begin{align}\label{55b}
 \widetilde{H}_1\leq e^{pv}\left[-p\frac{e^{\overline\gamma v}}{e^{\gamma u}}
+\gamma +C_4\gamma e^{\gamma u}e^{-pv}+ \delta_1 C_3 e^{-pv}\right]
\end{align}
and
\begin{align}\label{55c}
\widetilde{H}_2\leq e^{qu}\left[-q\frac{e^{\gamma u}}{e^{\overline\gamma v}}
+\overline\gamma
 +C_4\overline\gamma e^{\overline\gamma v}e^{-qu}+\delta_2 C_3e^{-qu}\right].
\end{align}

Assume that $\gamma, \overline\gamma>0$ satisfy
\begin{eqnarray}\label{56}
& & \overline\gamma=\gamma\frac{p}{q}.
\end{eqnarray}
Let the constants $C'_1,\,C'_2>0$ be given by
Lemma~\ref{30}. 
By (\ref{31}) and (\ref{56}), we then have
\begin{equation}\label{781}
\begin{array}{ll}
&(C'_1)^{\gamma/q}
\le \displaystyle\frac{e^{\gamma u}}{e^{\overline\gamma v}}
=\left(\frac{e^{qu}}{e^{pv}}\right)^{\gamma/q}
\le (C'_2)^{\gamma/q} \qquad \mbox{on}\:[T/2,\,T)\times(\rho_1,\,\rho_2).
\end{array}
\end{equation}
Using (\ref{55b})-(\ref{781}) and (\ref{310})-(\ref{311}), we deduce that
$$e^{-pv}\widetilde{H}_1\leq -p(C'_2)^{-\gamma/q}+\gamma +
 C_4\gamma e^{(C_2\gamma/q)-C_1} (T-t)^{1-(\gamma/q)}+\delta_1 C_3 e^{-C_1}(T-t)
$$
and
$$e^{-qu}\widetilde{H}_2\leq -q(C'_1)^{\gamma/q}+\overline\gamma +
 C_4\overline\gamma e^{(C_2\gamma/q)-C_1}(T-t)^{1-(\gamma/q)}+\delta_2 C_3 e^{-C_1}(T-t)
$$
on $[T/2,\,T)\times(\rho_1,\,\rho_2)$.
Taking $\gamma>0$ sufficiently small so that
$$ \gamma<p (C'_2)^{-\gamma/q},\qquad
\overline\gamma=\frac{q\gamma}{p}<q(C'_1)^{\gamma/q},\qquad \gamma<q$$
and then $T_1$ close enough to $T$, we get
$$\widetilde{H}_1\leq 0,
\;\; \widetilde{H}_2\leq 0
\quad\mbox{ on }\:[T_1,\,T)\times(\rho_1,\,\rho_2)$$
and the lemma follows from (\ref{52})--(\ref{55}).
\end{proof}

With Lemmas~\ref{30} and \ref{49} at hand, we can now conclude the proof of Theorem~\ref{th1}.
\smallskip

\begin{proof}[Proof of Theorem~\ref{th1}.]  Let $(u,\,v)$ be a solution of
system (\ref{e1}) satisfying the hypotheses of Theorem~\ref{th1} and assume for contradiction that there exists
$\rho_0\in(0,\,R)$ such that
\begin{eqnarray}\label{60}
\underset{t\rightarrow
T}{\limsup}\,(u(t,\,\rho_0)+v(t,\,\rho_0))=\infty.
\end{eqnarray}
Using the strong maximum principle, it follows easily from the assumptions (\ref{monot})
 and $u_\rho\not\equiv 0$ or $v_\rho\not\equiv 0$ that
\begin{equation}\label{urho1}
u_\rho,\, v_\rho<0\quad\hbox{ on $(0,\,T)\times(0,\,R)$.}
\end{equation}
Set $\rho_1=\rho_0/4,$ $\rho_2=\rho_0/2$ and let $J,$ $\overline{J}$, $T_1$ be given by  Lemma \ref{49}.
 Since $c(\rho_1)=c(\rho_2)=0$, we have
$J,$ $\overline{J}\leq0$ on
$(T_1,\,T)\times\{\rho_1,\rho_2\}$.
Taking $\eps > 0$ sufficiently small and using (\ref{urho1}),
we see that $J,$ $\overline{J}\leq0$ on
$\{T_1\}\times[\rho_1,\,\rho_2]$.
 Since the system (\ref{50}) is cooperative, we deduce from the maximum principle that $J,$ $\overline{J}\leq0$ on
$(T_1,\,T)\times[\rho_1,\,\rho_2].$
Consequently,
\begin{eqnarray*}
-u_\rho\hspace{-0,6cm}& &\geq\eps c(\rho)\, e^{\gamma u}
\quad\hbox{ on $(T_1,\,T)\times[\rho_1,\,\rho_2].$}
\end{eqnarray*}
 By integration, we obtain
\begin{eqnarray*}
 & &e^{-\gamma u(t,\rho_2)}\geq  \eps\gamma 
 \int_{\rho_1}^{\rho_2}\sin^2\left( \frac{\pi(\rho-\rho_1)}{\rho_2-\rho_1}\right)d\rho>0\quad\hbox{for all } T_1\leq t<T.
\end{eqnarray*}
It follows that $u(t,\,\rho_2)$ is bounded for $ T_1\leq t<T$,
and similarly $v(t,\,\rho_2)$ is bounded for $ T_1\leq t<T$.
Since $u_\rho,$ $v_\rho\leq0$, this leads to a contradiction
with (\ref{60}) and proves the theorem.
\end{proof}

\section{Type I estimates and proof of Theorem \ref{th2}}
\setcounter{equation}{0}

\def\eps{\varepsilon}
\def\beps{\bar\varepsilon}

As indicated in Introduction, we consider the more general (nonequidiffusive) $m$-system
\begin{equation}
 \label{msyst}
\partial_tu_i-\delta_i\Delta u_i=f_i(u_{i+1}),\qquad i=1,\dots,m,
\end{equation}
where $m\ge 2$, $\delta_i>0$ and
\begin{equation}
 \label{hypmsyst}
 \hbox{$f_i\in C^1([0,\infty))\cap C^2(0,\infty)$ are nonnegative functions with $f_i',f_i''\ge 0$.}
 \end{equation}
By convention, we set $f_{m+1}=f_1$, $u_{m+1}=u_1$, $\delta_{m+1}=\delta_1$.
We shall prove the following:
\begin{pro}\label{protypeI}
Let $\Omega$ be a smoothly bounded domain of $\mathbb{R}^n$, $T\in (0,\infty]$ and assume (\ref{hypmsyst}).
Let $U=(u_1,\dots,u_m)$ be a nonnegative classical solution of (\ref{msyst}) on
$Q_T:=(0,T)\times\Omega$, under Dirichlet or Neumann boundary conditions,
with $U$ nonstationary and time-nondecreasing.
In case of Dirichlet boundary conditions assume in addition that $f_i(0)=0$ for $i=1,\cdots,m$.
Then, for any $t_0\in (0,T)$, there exists $\eps>0$ such that, for all $ i=1,\dots,m,$ we have
$$\partial_tu_i \ge\eps f_i(u_{i+1}) \quad\hbox{in $(t_0,T)\times\Omega$.}$$
\end{pro}

\begin{res}\label{rem51}
\rm{
(a) Proposition~\ref{protypeI} remains valid for $\Omega=\mathbb{R}^n$
if it is further assumed that $U$ satisfies initial conditions
$U(0)=U_0$ with $U_0\in BC^2(\mathbb{R}^n)$ and there exists $\hat\eps\in(0,1)$ such that
$$\delta_i\Delta u_{0,i}+(1-\hat\eps) f_i(u_{0,i+1})\ge 0, \quad i=1,\dots,m.$$
(This follows from a simple modification of the proof.)

(b) In case of Dirichlet boundary conditions, if the blow-up set of $U$ is a compact subset of $\Omega$,
we may relax the assumption $f_i(0)=0$
(by working on a subdomain of $\Omega$ in the proof below, as in \cite{friedman}).}
\end{res}

Now consider the model system
\begin{equation}
 \label{msyst2}
\partial_tu_i-\delta_i\Delta u_i=u_{i+1}^{p_i},\quad i=1,\dots,m,
\end{equation}
where $m\ge 2$, $\delta_i>0$ and $p_i>1$.
As a consequence of Proposition~\ref{protypeI}, we can for instance obtain the following result,
which extends the type~I estimates for time-increasing solutions in \cite{Wa} to the case of unequal diffusivities.
We note that for suitable range of exponents $p_i$ (namely, the Fujita-subcritical range),
the type~I estimate 
 was proved for any positive blow-up solution in \cite{FiQu} (see also \cite{CaMi, AHV}
for earlier results in this direction).
Although the result there is given for $\delta_i=1$, the proof covers the case of unequal diffusivities as well.
On the other hand, single-point blow-up results for system (\ref{msyst2}) with possibly unequal diffusivities
were recently obtained in \cite{Mah}, by adapting the techniques in \cite{souplet, MST}.

\begin{pro}
Let $\Omega$ be a smoothly bounded domain of $\mathbb{R}^n$.
Let $U=(u_1,\dots,u_m)$ be a nonnegative classical solution of (\ref{msyst2}),
under Dirichlet or Neumann boundary conditions.
Assume that $U$ is nonstationary and time-nondecreasing.
Then $U$ blows up at a finite time $T>0$ and there exists $C>0$ such that, for all $i=1,\dots,m$, we have
$$u_i(t,x)\le C(T-t_i)^{-\alpha_i}  \quad\hbox{on $(0,T)\times\Omega$,}$$
where
$$\alpha_i=\displaystyle\frac{1+p_i+\displaystyle\sum_{l=i+1}^{m+i-2} p_i\cdots p_l}{p_1\cdots p_m-1}$$
(with the convention $p_{i+m}=p_i$).
\end{pro}

The proof is a direct consequence of Proposition~\ref{protypeI} and of the arguments in \cite{Wa}
 (see also \cite{Mah}), and is hence ommitted.

\begin{proof}[Proof of Proposition~\ref{protypeI}.]
Set
$$J_i=\partial_t u_i-\eps_if_i(u_{i+1}),$$
with $\eps_i\in (0,1)$, $i=1,\cdots,m$.
We make the convention $f_{i+m}=f_i$, $u_{i+m}=u_i$, etc.
Note that
 $J_i\in C((0,\,T)\times\overline\Omega)\cap
W^{1,2;k}_{loc}(Q_T)$, for all $1<k<\infty$,
 by parabolic $L^p$-regularity.
We compute (a.e. in $Q_T$):
$$
 \begin{array}{ll}
\partial_tJ_i-\delta_i\Delta J_i
&=\bigl[(\partial_t-\delta_i\Delta)u_i\bigr]_t-\eps_i\bigl(\partial_t-\delta_i\Delta\bigr)f_i(u_{i+1}) \\
\noalign{\vskip 2mm}
&= f'_i(u_{i+1})\partial_t u_{i+1}-\eps_i  f'_i(u_{i+1})\partial_t u_{i+1} \\
\noalign{\vskip 2mm}
&\qquad+\eps_i \delta_i f'_i(u_{i+1})
\Delta u_{i+1}+\eps_i\delta_i f''_i(u_{i+1})|\nabla u_{i+1}|^2 \\
\noalign{\vskip 2mm}
&\ge  f'_i(u_{i+1})\Bigl[(1-\eps_i)\partial_t u_{i+1}+\eps_i\delta_i\delta_{i+1}^{-1}\bigl(\partial_t u_{i+1}-f_{i+1}(u_{i+2})\bigr)\Bigr]\\
\noalign{\vskip 2mm}
&=  f'_i(u_{i+1})\Bigl[\bigl(1-\eps_i+\eps_i\delta_i\delta_{i+1}^{-1}\bigr)\partial_t u_{i+1}
-\eps_i\delta_i\delta_{i+1}^{-1}f_{i+1}(u_{i+2})\Bigr] \\
\noalign{\vskip 2mm}
&= c_if'_i(u_{i+1})
\left[\displaystyle{1-\eps_i+\eps_i\delta_i\delta_{i+1}^{-1}\over \eps_i\eps_{i+1}^{-1}\delta_i\delta_{i+1}^{-1}}
\,\partial_t u_{i+1}
-\eps_{i+1}f_{i+1}(u_{i+2})\right],
\end{array}
 $$
with $c_i=\eps_i\eps_{i+1}^{-1}\delta_i\delta_{i+1}^{-1}$.

Since $\partial_tu_{i+1}\ge 0$, we thus have
\begin{equation}
 \label{msyst-eq2}
\partial_tJ_i-\delta_i\Delta J_i\ge c_i f'_i(u_{i+1}) J_{i+1},\quad i=1,\dots,m,
\end{equation}
provided
$$1-\eps_i+\eps_i\delta_i\delta_{i+1}^{-1}\ge \eps_i\eps_{i+1}^{-1}\delta_i\delta_{i+1}^{-1}.$$
This is equivalent to
$$\delta_{i+1}\delta_i^{-1}(1-\eps_i)\ge \eps_i(\eps_{i+1}^{-1}-1)$$
that is,
$$\delta_i^{-1}(\eps_i^{-1}-1)\ge \delta_{i+1}^{-1}(\eps_{i+1}^{-1}-1)$$
for all $i=1,\dots,m$.
Since $\eps_{m+1}=\eps_1$ and $\delta_{m+1}=\delta_1$, this is satisfied if and only if
$$\delta_i^{-1}(\eps_i^{-1}-1)= \delta_{1}^{-1}(\eps_{1}^{-1}-1),\quad i=2,\dots,m.$$
Thus, for any given $\eps_1\in (0,1)$, there is a (unique) admissible choice of $\eps_{i} \in (0,1)$, $i=2,\dots,m$, given by
\begin{equation}
 \label{msyst-eq3}
\eps_{i}=\eps_{i}(\eps_1):=\bigl[1+\delta_{i}\delta_1^{-1}(\eps_1^{-1}-1)\bigr]^{-1}.
\end{equation}

The rest of the proof is then similar to, e.g., \cite[Theorem~23.5]{pavol}.
Namely, under our assumptions, we have $J_i=0$ (or $\partial_\nu J_i=0$) on $(T/2,T)\times \partial\Omega$.
 Since by assumption $\partial_tu_i\ge 0$ and $\partial_tu_i\not\equiv 0$, by using the strong maximum principle (and the Hopf Lemma in the case of Dirichlet conditions),
we also have $J_i(T/2,\cdot)\ge 0$ provided $\eps_i\in (0,\eta)$ with $\eta>0$ sufficiently small.

Now noting that $\eps_{i}(\eps_1)\to 0$ as $\eps_1\to 0$, we may then choose $\eps_i\in (0,\eta)$ satisfying (\ref{msyst-eq3}).
Applying the maximum principle to the system (\ref{msyst-eq2}) for the $J_i$ (note that this system is cooperative due to $f'_i\ge 0$), we deduce that $J_i\ge 0$, which implies the result.
\end{proof}

\medskip

We finally give the:

\begin{proof}[Proof of Theorem \ref{th2}.]
(i)  Let $T\in (0,\infty]$ be the maximal existence time of $U$ and pick $t_0\in (0,T)$.
Applying Proposition~\ref{protypeI} (or Remark~\ref{rem51}(a)), we deduce that
$$u_t\ge \eps f(v)\ge \eps(e^{pv}-1),\qquad v_t\ge g(u)\ge \eps(e^{qu}-1)$$
 in $Q:=(t_0,T)\times\Omega$.
Let $c=\eps\min(p,q)$. Adding up, we get
\begin{equation}
 \label{eqpfthm2}
(qu+pv)_t
\ge c (e^{qu}+e^{pv}-2)
\ge 2c (e^{(qu+pv)/2}-1), \qquad (t,x)\in Q.
\end{equation}
Since $u, v>0$ in $Q$ by the strong maximum principle,
we see from (\ref{eqpfthm2}) that $T$ must be finite.

Integrating (\ref{eqpfthm2}), it follows that
$$H\bigl[(pu+qv)(t,x)\bigr]\ge c(T-t),\quad\hbox{ where }
H(X)=\int_{X}^\infty\frac{ds}{e^{s/2}-1}.$$
Since there exists a constant $C>0$ such that
$$H(X)\le Ce^{-X/2}, \quad X\ge 1,$$
we deduce that, for each $(t,x)\in Q_T$ we have either  $C\exp\bigl[-\frac12(pu+qv)(t,x)\bigr]\ge c(T-t)$
or $(pu+qv)(t,x)\le 1$.
Therefore, there exists a constant $c_2>0$ such that
$$e^{qu}e^{pv}\le  c_2(T-t)^{-2}.$$

Finally, we observe that $U:=e^{qu}$ satisfies
$$U_t-\delta_1\Delta U=qe^{qu}(u_t- \delta_1\Delta u- q\delta_1|\nabla u|^2)\le qe^{qu} e^{pv}\le q c_2(T-t)^{-2}$$
in $Q_T$, and similarly for $V:=e^{pv}$. Comparison with the supersolution $\bar U:=M(T-t)^{-1}$
for $M$ sufficiently large yields the conclusion.

\smallskip

(ii) This follows immediately from assertion (i) and Theorem \ref{th1}.
\end{proof}

\begin{re}\label{rem51a}
\rm{
(a) In the case of power nonlinearities, for any radially symmetric nonincreasing solution,
a simple argument using the first eigenfunction
yields the type I estimate away from the origin (see \cite[Section~3]{MST}).
In the case of exponential nonlinearities, this argument gives
$$
u(t,x), v(t,x) \leq C\eps^{-n}(1+|\log(T-t)|),\quad 0<t<T,\ \eps\le |x|<R
$$
for every $\eps>0$, with some $C>0$, which is weaker than \eqref{estimateTypeI} and insufficient to work out the analysis in section~3.
\smallskip

(b) It is not clear if our results on single-point blow-up can be extended to systems
with exponential nonlinearities and more than two unknows, similar to (\ref{msyst2}).
The main difficulty is to extend the nondegeneracy property of
Proposition~\ref{15}, which seems nontrivial.}
\end{re}

\bibliographystyle{plain}

\end{document}